\documentclass[reqno, 11pt]{amsart}

\makeatletter
\let\origsection=\section \def\section{\@ifstar{\origsection*}{\mysection}} 
\def\mysection{\@startsection{section}{1}\z@{.7\linespacing\@plus\linespacing}{.5\linespacing}{\normalfont\scshape\centering\S}}
\makeatother  

\usepackage{mathrsfs}
\usepackage{amsmath}
\usepackage{amssymb}
\usepackage{amsthm}
\usepackage{mathtools}
\usepackage{letterswitharrows}
\usepackage{tikz-cd}
\usetikzlibrary{intersections,spath3}
\usepackage{thmtools}
\usepackage{thm-restate}
\usepackage[only,llbracket,rrbracket]{stmaryrd}
\usepackage[linesnumbered,ruled,vlined]{algorithm2e}
\usepackage{enumitem}
\setenumerate{label={\normalfont (\roman*)}}

\usepackage[utf8]{inputenc}
\usepackage[T1]{fontenc}
\usepackage{lmodern}
\usepackage[babel]{microtype}
\usepackage[english]{babel}
\usepackage{relsize}

\usepackage{graphicx}
\usepackage{subcaption}
\usepackage{import}

\usepackage{svg}
\usepackage[extractname=filename]{svg-extract}

\usepackage{geometry}
\usepackage{xcolor} 
\colorlet{darkishRed}{red!80!black}
\colorlet{darkishBlue}{blue!60!black}
\colorlet{darkishGreen}{green!60!black}
\definecolor{lucasBlue}{HTML}{A4D4F5}
\definecolor{lucasDarkGreen}{HTML}{8CBF65}
\definecolor{lucasLightRed}{HTML}{F5B3B9}
\definecolor{lucasOrange}{HTML}{EBC07A}

\usepackage{hyperref}
\hypersetup{
	colorlinks,
	linkcolor={blue!60!black},
	citecolor={green!60!black},
	urlcolor={blue!60!black},
}
\usepackage[abbrev, msc-links]{amsrefs}
\usepackage[nameinlink, capitalise, noabbrev]{cleveref}
\crefformat{enumi}{#2#1#3}
\crefformat{equation}{#2(#1)#3}
\crefname{mainresult}{Theorem}{Theorems}
\usepackage{doi}

\renewcommand{\PrintDOI}[1]{\doi{#1}}

\let\setminus=\smallsetminus

\renewcommand{\subset}{\subseteq}

\renewcommand{\leq}{\leqslant}
\renewcommand{\geq}{\geqslant}

\newtheorem{theorem}{Theorem}[section] 

\newtheorem{corollary}[theorem]{Corollary}
\newtheorem{lemma}[theorem]{Lemma}

\newcounter{claimcounter}[theorem]
\newtheorem{claim}[claimcounter]{Claim}
\newtheorem*{claim*}{Claim}
\crefname{claim}{Claim}{Claims}
\newcounter{casecounter}[theorem]
\newtheorem*{case*}{Case}
\crefname{case}{Case}{Cases}

\newcounter{subclaimcounter}[claimcounter]
\newtheorem*{subclaim*}{Subclaim}

\theoremstyle{definition}

\crefname{mainexample}{Example}{Examples}

\crefname{example}{Example}{Examples}

\crefname{routine}{Routine}{Routines}

\crefname{subroutine}{Subroutine}{Subroutines}

\crefname{subsubroutine}{Subsubroutine}{Subsubroutines}

\crefname{step}{Step}{Steps}
\theoremstyle{remark}

\crefname{subsection}{Subsection}{Subsections}

\newcommand{\COMMENT}[1]{{}}

\let\epsilon=\varepsilon
\let\theta=\vartheta
\let\rho=\varrho
\let\phi=\varphi
\def\N{\mathbb N}

\makeatletter

\def\calCommandfactory#1{%
  \expandafter\def\csname c#1\endcsname{\mathcal{#1}}}
\def\frakCommandfactory#1{%
  \expandafter\def\csname frak#1\endcsname{\mathfrak{#1}}}
   
\newcounter{ctr}
\loop
  \stepcounter{ctr}
  \edef\X{\@Alph\c@ctr}
  \expandafter\calCommandfactory\X
  \expandafter\frakCommandfactory\X
\ifnum\thectr<26
\repeat

\usepackage{etoolbox}

\newbool{arXiv}
\booltrue{arXiv} 

\newcommand{\arXivOrNot}[2]{\ifbool{arXiv}{{#1}}{{#2}}}

\newbool{pdfBool}
\booltrue{pdfBool} 

\newcommand{\pdfOrNot}[2]{\ifbool{pdfBool}{{#1}}{{#2}}}

\makeatletter
\newcommand\thankssymb[1]{\textsuperscript{\@fnsymbol{#1}}}
\makeatother

\newcounter{mylabelcounter}

\makeatletter
\newcommand{\labelText}[2]{%
#1\refstepcounter{mylabelcounter}%
\immediate\write\@auxout{%
  \string\newlabel{#2}{{1}{\thepage}{{\unexpanded{#1}}}{mylabelcounter.\number\value{mylabelcounter}}{}}%
}%
}

\makeatletter
\newcommand\footnoteref[1]{\protected@xdef\@thefnmark{\ref{#1}}\@footnotemark}
\makeatother

\definecolor{cMaroon}{HTML}{93152a}
\newcommand{\defn}[1]{{\color{cMaroon}{\emph{#1}}}}

\DeclareMathOperator{\obsii}{\mathcal{H}}

\usepackage{tabularx}

\title{On tree-decompositions for infinite chordal graphs}

\author[M.~Pitz]{Max Pitz}

\author[L.~Real]{Lucas Real\thankssymb{1}}
\thanks{\thankssymb{1} Supported by São Paulo Research Foundation (FAPESP) through grant number 2025/00669-5.}

\author[R.~Schaut]{Roman Schaut\thankssymb{2}}
\thanks{\thankssymb{2} Supported by a State Graduate Funding Program scholarship of the University of Hamburg.}

\address{University of Hamburg, Department of Mathematics, Bundesstrasse 55 (Geomatikum), 20146 Hamburg, Germany}
\email{\{max.pitz, roman.schaut, lucas.real\}@uni-hamburg.de}

\keywords{Chordal graphs; tree-decompositions; finite tree-width.}

\subjclass[2020]{05C63, 05C40, 05C62, 	05C69}

\begin{document}

\begin{abstract}
 
A graph is chordal if it contains no induced cycle of length four or more. While finite chordal graphs are precisely those admitting tree-decompositions into cliques, this fails for infinite graphs. We establish two results extending the known theory to the infinite setting.

Our first result strengthens sufficient conditions of Halin, K\v{r}\'i\v{z}-Thomas, and Chudnovsky-Nguyen-Scott-Seymour: We show that every chordal graph without a strict comb of cliques admits a tree-decomposition into maximal cliques.

Our second result characterises the chordal graphs admitting tree-decompositions into finite cliques: a connected graph admits such a decomposition if and only if it is chordal, admits a normal spanning tree, and does not contain $\mathcal{H}$---an infinite clique with two non-adjacent dominating vertices---as an induced minor. Combined with the characterisation of graphs with normal spanning trees, this yields a description by three types of forbidden minors.

Both proofs proceed via greedy constructions of length $\omega$, with the key new ingredient for the second result being an Extension Lemma that uses a finiteness theorem of Halin on minimal separators to produce suitable finite clique extensions at each step.

\end{abstract}

\vspace*{-12pt}
\maketitle

\section{Introduction}

A graph is \defn{chordal} if it has no induced cycle of length four or more.
The finite chordal graphs are precisely the graphs admitting tree-decompositions into (maximal) cliques; see e.g.~\cite[Proposition~12.3.6]{bibel}.
 For infinite graphs, however, this is no longer true. 
 While every graph with a tree-decomposition into cliques is chordal, some countable chordal graphs do not admit a tree-decomposition into cliques, as observed by Halin \cite{halin1984representation}. In a positive direction, Halin established the following useful sufficient condition for an infinite chordal graph to have a tree-decomposition into cliques:

\begin{theorem}[Halin 1984]
\label{thm_main_chordalHalin}
Every chordal graph without infinite clique has a tree-decomposition into maximal cliques. 
\end{theorem}

K{\v r}{\'i}{\v z} and Thomas \cite{kvrivz1990clique} observed that the conclusion of Theorem~\ref{thm_main_chordalHalin} extends to chordal graphs in which any two maximal cliques intersect in at most $k$ vertices, for some fixed $k \in \mathbb{N}$.

A \defn{comb of cliques} in a graph $G$ consists of a countably infinite clique $\{v_i \colon i \in \N\}$ together with, for each $i > 1$, a vertex adjacent in $G$ to all of $v_1, \ldots , v_{i-1}$ but not to $v_i$. 
Chudnovsky, Nguyen, Scott and Seymour \cite{chudnovsky2025vertex} unified Halin's and K{\v r}{\'i}{\v z} \& Thomas's sufficient conditions by showing that every chordal graph without a comb of cliques has a tree-decomposition into maximal cliques. 

Our first main result strengthens these sufficient conditions even further. 
A \defn{strict comb of cliques} in a graph $G$ consists of a countably infinite clique $\{v_i \colon i \in \N\}$ together with, for each $i > 1$, a vertex adjacent in $G$ to all of $v_1, \ldots , v_{i-1}$, but to none of the vertices $v_i,v_{i+1},v_{i+2},\ldots$.

\begin{theorem}
\label{thm_main_chordal}
Every chordal graph without a strict comb of cliques has a tree-decomposition into maximal cliques. 
\end{theorem}

We shall prove Theorem~\ref{thm_main_chordal} via a simple greedy construction of length $\omega$, similar in spirit to the existence proof of normal spanning trees in \cite{pitz2020unified}.

Still, it remains a well-known open problem to characterise which infinite chordal graphs do have a tree-decomposition into (maximal) cliques, as obtained by Diestel in \cite{SimplicialMinors} for the countable setting.
In fact, his work addresses the broader conjecture of describing the (countable) graphs admitting \textit{simplicial tree-decompositions into primes}, i.e., tree-decompositions whose adhesion sets are complete and whose bags induce graphs that are not separated by cliques.
Incidentally, stating a main question in the theory of simplicial decompositions (see Section 2 in \cite{DecomposingInfiniteGraphs}), it is even unknown when these bags can be obtained as finite graphs. 
Since prime induced subgraphs of chordal graphs are actually cliques, the second main result of this paper contributes towards this problem by characterising precisely which chordal graphs admit a tree-decomposition into finite cliques (note that the strict comb of cliques has such a tree-decomposition).

In order to state this description, we recall that a \defn{minor embedding} of a graph $H$ into a graph $G$ is a family $\{B_v: v \in V(H)\}$ of disjoint connected subgraphs of $G$ (called \defn{branch sets}) such that there is a $B_u{-}B_v$ edge in $G$ whenever $uv\in E(H)$. 
In this case, $H$ is said to be a \defn{minor} of $G$, while the subgraph induced by $\bigcup_{v\in V(H)}B_v$ is often referred to as the corresponding \textit{model} of $H$ in $G$.  
On the other hand, if there is a $B_u{-}B_v$ edge in this minor model \textit{precisely} when the edge $uv\in E(H)$ exists, then $H$ is called an \defn{induced minor}\footnote{When restricting the setting of Diestel from \cite{SimplicialMinors} to a chordal graph $G$, we remark that a minor $H$ of $G$ is \textit{simplicial} if it is induced in this sense and the connected components of the complement of the corresponding model have complete neighbourhoods.} of $G$.

As illustrated by \cref{fig:obsii}, let $\obsii$ be the graph obtained from an infinite clique by adding two non-adjacent vertices that dominate it. Within these terms, the second main result of this paper reads as follows:

\begin{theorem}
\label{thm_main_chordalfinite}
    A connected graph admits a tree-decomposition into finite cliques if, and only if, it is chordal, admits a normal spanning tree and does not contain $\obsii$ as an induced minor.
\end{theorem}

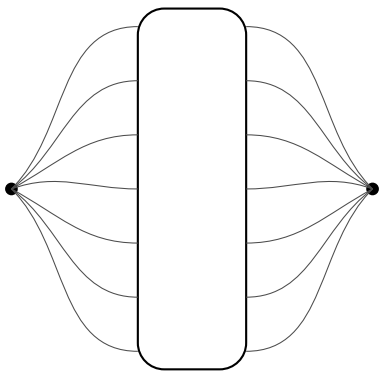
\begin{figure}
    \centering
    \begin{tikzpicture}[scale=1.2]

\coordinate (L) at (-2,0);
\coordinate (R) at ( 2,0);

\draw[thick, rounded corners=10pt]
  (-0.6,-2) rectangle (0.6,2);

\fill (L) circle (2pt);
\fill (R) circle (2pt);

\foreach \t in {-1.8,-1.2,-0.6,0,0.6,1.2,1.8}
{
    \pgfmathsetmacro{\ang}{20 + 25*abs(\t)/1.8}

    \pgfmathsetmacro{\sign}{ifthenelse(\t>=0,1,-1)}

    \pgfmathsetmacro{\leftangle}{\sign*\ang}

    \pgfmathsetmacro{\rightangle}{180 - \sign*\ang}

    \draw[black!70]
      (L) to[out=\leftangle, in=180] (-0.6,\t);

    \draw[black!70]
      (0.6,\t) to[out=0, in=\rightangle] (R);
}

\end{tikzpicture}
\caption{The graph $\obsii$, where the rectangle in the picture represents an infinite clique on which the two distinguished and non-adjacent vertices have their (common) neighbourhood.}
 \label{fig:obsii}
\end{figure}

Pitz \cite{pitz2021proof} characterised graphs with normal spanning trees via two types of forbidden minors. Combined with \cref{thm_main_chordalfinite}, this yields a characterisation in terms of three types of forbidden minors.

Again, our proof of Theorem~\ref{thm_main_chordalfinite} proceeds via a simple greedy construction of length $\omega$. In every step, the assumption on $\obsii$ is used to show that there is a suitable extension of the partial tree-decomposition into each currently remaining component. During this process, the normal spanning tree is used to argue that the whole graph is exhausted after $\omega$ steps.

Incidentally, as we shall later detail through \cref{lem:Necessity}, it is not difficult to see that both criteria from the characterisation provided by \cref{thm_main_chordalfinite} are in fact necessary for the existence of tree-decompositions into finite cliques.
For instance, the existence of a normal spanning tree in a connected graph $G$ is equivalent to $G$ having a tree-decomposition into finite parts. 
On the other hand, the graph $\obsii$ itself does not admit a tree-decomposition $(T,\mathcal{V})$ into finite cliques: otherwise, the two (non-adjacent) dominating vertices of $\obsii$ would live in different decomposition bags, thus needing to be separated from each other by some finite adhesion set of $(T,\mathcal{V})$.
Therefore, since the existence of tree-decompositions into (finite) cliques is inherited by induced minors (as ensured by Lemmas 12.3.2 and 12.3.3 in \cite{bibel}, for example), no graph admitting such decompositions can contain an induced minor copy of $\obsii$.

For background on tree-decompositions, we refer the reader to the textbook \cite[\S12]{bibel}.

\section{A sufficient condition for tree-decompositions into cliques}
\label{sec_2}

Given a graph $G$, write $N(v)$ for the neighbourhood of a vertex $v$ and, given $S \subset V(G)$, we abbreviate $N_S(v) = N(v) \cap S$ for the neighbours of $v$ in $S$.
We recall the following folklore lemma on clique separators in chordal graphs:

\begin{lemma}
\label{lem_chordal_1}
Let $K \subseteq G$ be a clique in a connected, chordal graph $G$.
Let $C$ be a component of $G - K$, and write $S = N(C)$ for the set of neighbours of $C$. 
Then, for every vertex $v \in C$ with $N_S(v) \subsetneq S$
there is $w \in C$ with  $N_S(v) \subsetneq N_S(w)$.
\end{lemma}

\begin{proof}
Consider a vertex $v \in C$ with $N_S(v) \subsetneq S$. In the subgraph of $G$ induced by $C \cup (S \setminus N(v))$, choose a shortest path $P$ from $v$ to $S \setminus N(v)$.
Let $s_w \in S \setminus N(v)$ be the last vertex of $P$, and let $w$ be the penultimate vertex on $P$ (which exists, since $v$ is not adjacent to $s_w$).

We claim that $w$ satisfies the assertion of the lemma.
Suppose otherwise. Then there is $s_v \in N_S(v)$ with $s_v \notin N(w)$.  
As $S$ is complete, the vertices $s_v$ and $s_w$ are adjacent, and so $vP w s_w s_v v$ is a cycle, so it has a chord.
By the minimality of $P$ it follows that $P$ has no chords; so every chord of this cycle is incident with $s_v$. Let $u$ be the neighbour of $s_v$ on $Pw$ that is closest to $w$. Then $u P s_ws_vu$ is a chordless cycle of length $\geq 4$, contradicting that $G$ is chordal.
\end{proof}

\begin{corollary}
\label{cor_chordal_2}
Let $K \subseteq G$ be a clique in a connected, chordal graph $G$, and $C$ a component of $G-K$ with neighbourhood $S = N(C)$.
If $G$ contains no strict comb of cliques, then $C$ contains a vertex $v$ with $N_S(v) = S$.
\end{corollary}

\begin{proof}
If there is no such vertex $v$, then by Lemma~\ref{lem_chordal_1} there exist vertices $w_1,w_2,w_3,\ldots$ in $C$ with
$N_S(w_1) \subsetneq N_S(w_{2}) \subsetneq \cdots.$
Pick $v_1 \in N_S(w_1) $ and $v_{i+1} \in N_S(w_{i+1}) \setminus N_S(w_{i})$ for all $i \geq 1$.
Then $\{v_i \colon i \in \N\} \subseteq S \subseteq K$ induces a countably infinite clique, and for each $i > 1$, the vertex $w_{i-1}$ is adjacent to all of $v_1, \ldots , v_{i-1}$, but to no vertex of $v_i,v_{i+1},v_{i+2},\ldots$,
 yielding a strict comb of cliques, a contradiction.
\end{proof}

\begin{proof}[Proof of \cref{thm_main_chordal}]
We may assume that the graph is connected.
Our plan is to define the desired tree-decomposition $(T,\cV)$ level by level. 
    We do this by constructing induced subgraphs $G^0 \subset G^1 \subset \dots$ of $G$ together with tree-decompositions $(T^i, \cV^i)$ of $G^i$ into distinct maximal cliques of $G$, so that the tree-decompositions extend each other in the sense that
    \begin{enumerate}[label=(\roman*)]
    \item \label{item:thm1:increasingTd}$T^0 \subseteq T^1 \subseteq T^2 \subseteq \dots$,
    \item \label{item:thm1:bagswelldefined} $V_t^i = V_t^j$ for all $i \leq j$ and $t \in T^i$.
    \end{enumerate}
    In particular, from now on we simply write $V_t$ for the part corresponding to a vertex $t \in T^i$. 
    To keep the construction going, we also ensure inductively that
    \begin{enumerate}[resume, label=(\roman*)]
    \item \label{item:thm1:compLeaf} for every component $C$ of $G-G^i$, its neighbourhood $N(C)$ is included in a part $V_{t_C}$ for a unique node $t_C \in T^i$ at height $i$.
    \end{enumerate}
    
    For the base case $i=0$, let $G^0$ be any maximal clique of $G$. Let $T^0$ be the tree on a single node $r$ and put $V_r := V(G^0)$.
        
    Next, assume that for some $i \in \N$ we have already constructed a tree-decomposition~$(T^i, \cV^i)$ of $G^i$ into maximal cliques. Let $C$ be a component of $G-G^i$ and write $S_C = N(C)$. By \cref{item:thm1:compLeaf}, the set $S_C$ is included in a (maximal) clique $V_{t_C}$ for a unique node $t_C \in T^i$ at height $i$. By \cref{cor_chordal_2}, there is a vertex $v \in C$ with $S_C \subseteq N(v)$. Extend $\{v \} \cup S_C$ to a maximal clique $K_C$ in $G$.
 Let $T^{i+1}$ be the tree obtained from $T^i$ by adding, for each component $C$ of $G-G^i$, a new successor $s_C$ to $t_C$, with corresponding part $V(K_C)$. 
Let $G^{i+1}$ be the subgraph of $G$ induced by all parts of $T^{i+1}$. It is clear that $(T^{i+1}, \cV^{i+1})$ is a tree-decomposition of $G^{i+1}$ satisfying \cref{item:thm1:increasingTd}--\cref{item:thm1:compLeaf}. 
 
      The desired tree-decomposition\ for \cref{thm_main_chordal} will be $T := \bigcup_{i \in \N} T^i$ with the parts $V_t$ for $t \in T$, which is well-defined due to \cref{item:thm1:increasingTd} and \cref{item:thm1:bagswelldefined}. Clearly, $(T, \cV)$ is a tree-decomposition\ of $\bigcup_{i \in \N} G^i$. To complete the proof, we argue that $\bigcup_{i \in \N} G^i = G$.    
      
    Suppose for a contradiction that $\bigcup_{i \in \N} G^i \neq G$. 
    Then there is a component $D$ of $G - \bigcup_{i \in \N} G^i$.
 For each $i \in \N$ let $D_i$ be the unique component of $G - G^i$ with $D_i \supseteq D$. 
    By \cref{item:thm1:compLeaf}, the neighbourhood $N(D_i)$ is included in $V_{t_i}$ for a unique node $t_i$ of $T$ at height $i$, and hence is a clique. It follows that $N(D)$ is a clique. Since $N(D) \setminus G^i \neq \emptyset$ for all $i \in \N$ (as otherwise, $D = D_i$, contradicting that in our construction, we extended the tree-decomposition into $D_i$), we have $|N(D)| = \infty$. This would complete a proof of Halin's \cref{thm_main_chordalHalin}. For \cref{thm_main_chordal}, we need one more argument.

    Choose a sequence $\{v_i \colon i \in \N\} \subseteq N(D)$ as follows: Let $v_1 \in N(D)$ be arbitrary, and suppose we have already selected $v_1,\ldots, v_{n-1}$. 
    Let $i_n \in \N$ be minimal so that $v_1,\ldots,v_{n-1} \in G^{i_n}$, and pick $v_{n} \in N(D) \setminus G^{i_n}$ arbitrarily. 
        Then the vertices $\{v_i \colon i \in \N\}$ induce a countably infinite clique in $G$. 
        We claim that for each $n > 1$, there is a vertex adjacent to all of $v_1, \ldots , v_{n-1}$ but not to $v_n,v_{n+1},v_{n+2},\ldots$. 
        Indeed, we have $v_1, \ldots , v_{n-1} \in N(D_{i_n})$, and so $v_1, \ldots , v_{n-1}$ all belong to $V_{t_{i_n}}$ but $v_n$ does not. 
    As $v_n$ is adjacent to all $v_1, \ldots , v_{n-1}$ (since $N(D)$ is complete), but is not contained in the maximal clique $V_{t_{i_n}}$, it follows that there is a vertex $w \in V_{t_{i_n}}$ which is adjacent to all $v_1, \ldots , v_{n-1} $ but not to $v_n$. 
    We claim that $w$ is adjacent to none of $v_{n+1}, v_{n+2}, \ldots$. Then we have found a strict comb of cliques, contrary to our assumption. 
    Towards this end, suppose for a contradiction that $w$ is adjacent to $v_{n+k}$ for some $k \geq 1$. 
    Then $w,v_1, \ldots , v_{n-1} \in N(D_{i_{n+k}}) \subseteq V_{t_{i_{n+k}}}$, and by choice of $i_{n+k}$, the clique $V_{t_{i_{n+k}}}$ also contains $v_n$; so $v_n$ is adjacent to $w$, giving the desired contradiction.  
\end{proof}

\section{Tree-decompositions into finite cliques}

In this section we prove our second main result, which we restate here for convenience:

\begin{theorem}\label{thm:TreeDecompositionConstruction}
    A connected graph $G$ admits a tree-decomposition into finite cliques if and only if it is chordal, admits a normal spanning tree and does not contain $\obsii$ as an induced minor.
\end{theorem}

As briefly mentioned in the introduction, the forward direction of this characterisation follows from routine arguments regarding tree-decompositions inherited by minors. 
For completeness, we provide the formal verification:

\begin{lemma}\label{lem:Necessity}
    Suppose that $G$ is a graph admitting a tree-decomposition into finite cliques. 
    Then, $G$ is chordal, has a normal spanning tree and does not contain $\obsii$ as an induced minor.
\end{lemma}
\begin{proof}
Every connected graph that admits a tree-decomposition into finite parts has a normal spanning tree, as one can verify by a routine application of Jung's \textit{Normal Spanning Tree Criterion} \cite[Satz 6']{jung1969wurzelbaume}; see also \cite[Theorem 2.2]{LinkedTDInfGraphs} for a detailed proof.
It is also easy to see that if a graph admits a tree-decomposition into cliques then it is chordal.

Thus, we focus here on the (non-existence) of $\obsii$ as an induced minor. 
    We first claim that $\mathcal{H}$ admits no tree-decomposition into finite cliques.
    Indeed, let $(T,\cV)$ be a tree-decomposition of $\mathcal{H}$ and let us assume for a contradiction every bag of $(T,\cV)$ induces a finite clique.
    Let $a_1$ and $a_2$ be the two dominating vertices of $\mathcal{H}$.
    Since $a_1$ and $a_2$ are not adjacent, but every bag of $(T,\cV)$ is complete, no bag contains both $a_1$ and $a_2$.

    But then let $P \subseteq T$ be a shortest path between two nodes $t_1,t_2\in T$ whose corresponding bags $V_{t_1}$ and $V_{t_2}$  contain $a_1$ and $a_2$, respectively. 
    Due to the above observation, $P$ has at least one edge $e$.
    By minimality of $P$, $V_e$ does neither contain $a_1$ nor $a_2$, and thus is a finite set separating $a_1$ from $a_2$ (by \cite{bibel}*{Lemma~12.3.1}).

    This contradicts the fact that $a_1$ and $a_2$ are joined by infinitely many pairwise disjoint paths in $\mathcal{H}$.
    
    Now, assume for a contradiction that $\{H_h \colon h \in \mathcal{H}\}$ comprises a model of an induced minor of $\mathcal{H}$ in $G$. Let $H$ be the subgraph of $G$ induced by $\bigcup_{h \in V(\obsii)} H_h$.
    Firstly, due to \cite{bibel}*{Lemma~12.3.2}, the pair $(T, (V_t \cap V(H))_{t \in T})$ is a tree-decomposition of $H$. 
        And this tree-decomposition of $H$ is still into finite cliques, since every bag is a subset of a bag of $(T,\cV)$.
    Hence, we may assume without loss of generality that $G=H$.
    Now, let $f:V(G) \longrightarrow V(\mathcal{H})$ be the map assigning to each vertex of $G$ the index of the branch set of $\{H_h \colon h \in\mathcal{H}\}$ containing it.
    For each $t \in T$ let $W_t := \{ f(v) \colon v \in V_t\}$.
    Then by \cite{bibel}*{Lemma~12.3.3}, the pair $(T,\mathcal{W})$ with $\mathcal{W}:=(W_t)_{t \in T}$ is a tree-decomposition of $\mathcal{H}$.
    Moreover, every bag of $(T,\mathcal{W})$ is finite, since every bag of $(T,\cV)$ was finite.
    But also every bag of $(T,\mathcal{W})$ comprises a finite clique: after all, given $t \in T$ and $x \neq y \in W_t$, we can write $x=f(v_x)$ for some $v_x \in V_t$ and $y=f(v_y)$ for some $v_y \in V_t$.
    Since $v_x$ and $v_y$ are adjacent, and $\{H_t \colon t \in \mathcal{H}\}$ is a model of an induced minor, also $x$ and $y$ are adjacent.
    This contradicts the fact that $\mathcal{H}$ admits no tree-decomposition into finite cliques.
\end{proof}

From now on to the end of this paper, and especially within \cref{subsec:extension}, we shall work towards establishing the converse direction of \cref{thm:TreeDecompositionConstruction}.
To that aim, however, \cref{subsec:preliminaries} first compiles the main technical lemmas about (minimal) vertex separators in chordal graphs.

\subsection{Preliminaries about separation properties}
\label{subsec:preliminaries}
This subsection closely follows the notation of Diestel in \cite{SimplicialMinors} for a fixed graph $G$.
Given a set of vertices $S\subseteq V(G)$, a connected component $C$ of $G - S$ is \defn{attached} (resp. \defn{unattached}) to a set $A\subseteq S$ if $A\subseteq N(C)$ (resp. $A\setminus N(C)\neq \emptyset$).
In the particular case where $S$ is a clique and $C$ is attached to $S$, the pair $(C,S)$ is called a \defn{side} of $G$. 
We shall endow the set of all sides of $G$ with the following partial order $\preceq$, as formalized by Lemma 3 in \cite{diestel199093}: \defn{$(C,S)\preceq (C',S')$} if, and only if, $C\subseteq C'$.

In its turn, given two disjoint sets of vertices $A,B\subseteq V(G)$, we define their set of \defn{crossing edges} \defn{$E(A,B)$}$:=\{uv\in E(G): u\in A, v \in B\}$. 
We say that a path $P$ in $G$ connecting a vertex $a\in A$ to a vertex $b\in B$ is an \defn{$A{-}B$ path} (or an \defn{$a{-}b$ path}) if $P\cap A =\{a\}$ and $P\cap B=\{b\}$.
In a dual sense, a set $S\subseteq V(G)$ is called an \defn{$A{-}B$ separator} if there exists no $A{-}B$ path in $G-S$.
In the particular case where $A=\{a\}$ and $B=\{b\}$ are singletons, the \defn{$a{-}b$ separator} $S$ is further required to contain neither $a$ nor $b$. 
When $S$ is $\subseteq{-}$minimal with this property, that is for every proper subset $S' \subsetneq S$ there exists an $A{-}B$ path in $G-S'$, we refer to  $S$ as a \defn{minimal $A{-}B$ separator}.

When $A$ and $B$ are connected and disjoint and have no crossing edges, the following lemma is enough for ensuring the existence of (minimal) $A{-}B$ separators:

\begin{lemma}\label{lem:ExistenceMinimalSeparators}
    Let $A,B \subseteq V(G)$ be two disjoint and connected sets of vertices such that $E(A,B)=\emptyset$.
    Then there exists a minimal $A{-}B$ separator $S$ that is disjoint from $A\cup B$.
    
    Moreover, an $A{-}B$ separator $S$ that is disjoint from $A \cup B$ is minimal if, and only if, the unique components of $G- S$ containing $A$ and $B$ respectively, are attached to $S$.
\end{lemma}
\begin{proof}
    We shall first prove the `Moreover'-part.
    Let $S$ be an $A{-}B$ separator disjoint from $A$ and $B$ and let $C_A$ and $C_B$ denote the components of $G-S$ containing $A$ and $B$, respectively.
    If $S$ is minimal, let us assume for a contradiction that $C_A$ does not fully attach into $S$.
    But then $N(C_A) \subsetneq S$ is a proper subset of $S$ that still separates $A$ from $B$ in $G$, a contradiction.
    On the other hand, let us assume that both $C_A$ and $C_B$ are attached to $S$.
    Then, for every $s \in S$, one can easily construct an $A{-}B$ path which hits $S$ precisely in $s$.
    Hence, no proper subset of $S$ can separate $A$ from $B$.
    
    For the main statement, let $S':=N(A)$.
    By assumption, $S'$ is disjoint from both $A$ and $B$, besides separating $A$ from $B$ in $G$.
    Let $C'_A$ be the component of $G-S'$ containing $A$.
    Then $N(C'_A)=N(A)=S'$.
    Now, let $C'_B$ be the component of $G-S'$ containing $B$ and let $S:= N(C'_B)$.
     In particular, since $A\subseteq C_A'$ and $B\subseteq C_B'$, the set $S$ is disjoint from both $A$ and $B$ by definition.
     We claim that $S$ is a minimal $A{-}B$ separator.

    By the `Moreover'-part, it suffices to prove that both $C_A$ and $C_B$ are attached to $S$.
    Indeed, $C_B$ attaches to $S$ again by definition of $S$.
    To see that $C_A$ attaches to $S$ as well, let us observe that $S \subseteq S'=N(C'_A)$ and hence $C_A \supseteq C'_A$.
    It follows that $N(C_A)=S$.
\end{proof}

In our special setting where the given graphs are chordal, and as first observed by Dirac in \cite{MR130190}, minimal $A{-}B$ separators are actually cliques:

\begin{corollary}\label{cor:MinimalSeparatorsIsClique}
    Let $G$ be a chordal graph and let $A,B \subseteq V(G)$ be two disjoint and connected sets such that $E(A,B)=\emptyset$.
    Then any minimal $A{-}B$ separator $S$ that is disjoint from $A \cup B$ must be a clique.
\end{corollary}
\begin{proof}
    Let us assume for a contradiction $x \neq y \in S$ are two non-adjacent vertices.
    By the `Moreover'-part of \cref{lem:ExistenceMinimalSeparators}, there exist two distinct components $C_A$ and $C_B$ in $G-S$ that attach to $S$.
    Since $C_A$ attaches to $S$, there exists a shortest $x{-}y$ path $P_A$ in $G[C_A \cup \{x,y\}]$ and similarly since $C_B$ attaches to $S$, a shortest $x{-}y$ path $P_B$ in $G[C_B \cup \{x,y\}]$.
    Moreover, $|P_A|,|P_B|\geq 2$ and since both $P_A$ and $P_B$ were picked as shortest paths, they are induced paths in $G$.
    But then $P_A P_B$ defines an induced cycle of length at least $4$ in $G$, a contradiction to chordality.
\end{proof}

 In terms of the existence of minimal $A{-}B$ separators, forbidding the graph $\obsii$ as an induced minor has now the following interpretation:
\begin{lemma}\label{cor:H2obstruction}
    A chordal graph $G$ admits no embedding of $\obsii$ as an induced minor if and only if, for every two disjoint and connected sets $A,B \subseteq V(G)$ with $E(A,B)=\emptyset$, any minimal $A{-}B$ separator disjoint from $A\cup B$ is finite. 
\end{lemma}
\begin{proof}
    Fix two disjoint and connected sets $A,B\subseteq V(G)$ with $E(A,B) = \emptyset$. According to \cref{cor:MinimalSeparatorsIsClique}, any minimal $A{-}B$ separator $S$ that is disjoint from $A \cup B$ is a clique.
    Moreover, if $C_A$ and $C_B$ denote the components of $G-S$ containing $A$ and $B$ respectively, then both $C_A$ and $C_B$ are attached to $S$ by \cref{lem:ExistenceMinimalSeparators}.
    In other words, if $S$ is infinite, $C_A \cup S \cup C_B$ describes an embedding of $\obsii$ into $G$ as an induced minor as follows: the vertices of $S$ are their own branch sets, while $C_A$ and $C_B$ constitute the branch sets of the two dominating vertices. 
    Since each vertex $s \in S$ is adjacent to both $C_A$ and $C_B$, this gives indeed a minor model of $\obsii$.
    And since $C_A$ and $C_B$ are distinct components of $G - S$, there are no edges between $C_A$ and $C_B$ in $G$ and, hence, this minor description is also induced.

    Conversely, let us assume a given chordal graph $G$ admits $\obsii$ as an induced minor and let $H_a$ and $H_b$ denote the branch sets of the two unique dominating vertices. 
    Then $H_a$ and $H_b$ are disjoint and connected sets, while $E(H_a,H_b) = \emptyset$ due to the property of $\obsii$ being an induced minor.
    But any (minimal) $H_a{-}H_b$ separator $S$ that is disjoint from $H_a \cup H_b$ must be infinite, as witnessed by the infinitely many pairwise disjoint $H_a{-}H_b$ paths that can be obtained inside the minor model of~$\obsii$.
\end{proof}

\subsection{Proof of \cref{thm:TreeDecompositionConstruction}}\label{subsec:extension} We are now able to construct a tree-decomposition as claimed by \cref{thm:TreeDecompositionConstruction}, which will require a suitable iterative step for constructing its bags recursively. 
In this setting, our approach relies on the following result by Halin regarding the number of distinct finite minimal separators between two vertices:

\begin{lemma}[\cite{HALIN199297}, Theorem 1]\label{lem:HalinResult}
    Let $a,b \in V(G)$ be two distinct and non-adjacent vertices of a graph $G$.
    If there is an infinite set of finite minimal $a{-}b$ separators, then there is one infinite such minimal $a{-}b$ separator.
\end{lemma}

For graphs with no induced $\obsii$-minor, \cref{lem:HalinResult} guarantees that each pair of non-adjacent vertices admits only finitely many minimal separators.

\begin{lemma}[Extension Lemma]
\label{lem:finiteCliqueExt}
Let $G$ be a chordal graph admitting no embedding of $\obsii$ as an induced minor, and let $(C,A)$ be a finite side in $G$. 
Then for each $v_* \in V(C)$, there exists a finite clique $B \supseteq A$ such that either $v_* \in B$, or the unique component of $G-B$ containing $v_*$ is unattached to $A$.
\end{lemma}

\begin{proof}
    Suppose that no finite clique $B\supseteq A$ contains $v_*$, or, equivalently, that $ A\setminus N(v_*)\neq \emptyset$.
    Then, for each $a \in A\setminus N(v_*)$, let $\mathcal{S}_a$ denote the set of all minimal $v_*{-}a$ separators containing neither $a$ nor $v_*$. Next, define $\mathcal{S} = \bigcup_{a \in A\setminus N(v_*)} \mathcal{S}_a$.
     We recall from \cref{cor:H2obstruction} that $\mathcal{S}$ is a set of finite cliques, since $G$ does not contain $\obsii$ as an induced minor. Moreover, each $\mathcal{S}_a$ is finite by \cref{lem:HalinResult}, and since $A$ is finite, $\mathcal{S}$ is finite, too.
    Then, let $(C_*,S_*)$ be a $\preceq{-}$maximal pair from the finite set $$\mathcal{C}:=\{(D,S): S \in \mathcal{S}, \, D\text{ is the connected component of }G- S\text{ containing }v_*\}.$$
    Let $a_*\in A$ be such that $S_*$ is a minimal $v_*{-}a_*$ separator.
    Since $A$ is a clique, it follows that $C_*\cap A = \emptyset$.
    
    Now, it suffices to prove that $S_*\cup A$ is a clique: after all, the component of $G- (S_*\cup A)$ containing $v_*$ is $C_*$, and, witnessing that $C_*$ is unattached to $A$, the vertex $a_*$ has no neighbour in $C_*$.

    Therefore, suppose for a contradiction that $S_*\cup A$ is not complete.
    Since both $S_*$ and $A$ are cliques, it follows that some $s_* \in S_*$ is not a neighbour of some $a\in A$.
    By \cref{cor:H2obstruction}, there exists a minimal finite clique $S\subseteq V(G)\setminus \{s_*,a\}$ separating $s_*$ and $a$. 
   
    \begin{claim}\label{claim:SepEmpty}
        $S\cap C_* = \emptyset$.
    \end{claim}
    \begin{proof}[Proof of Claim 1]\renewcommand{\qedsymbol}{$\blacksquare$}
    Suppose for a contradiction there exists a vertex $s'\in S\cap C_*$.
    Since $S$ is a minimal $s_*{-}a$ separator, it follows by \cref{lem:ExistenceMinimalSeparators} that the component $C'$ of $G-S$ containing $a$ is attached to $S$ and hence there exists a $s'{-}a$ path $P'$ in $C'\cup S$ intersecting $S$ in exactly $s'$.
    Also, there exists a vertex $p\in P' \cap S_*$ (because $s'\in C_*$ and $A\cap C_*=\emptyset$).
    Since $S_*$ is a clique, the concatenation $s_*pP'a$ provides a well-defined $s_*{-}a$ path that does not meet $S$, contradicting the fact that $S$ separates $s_*$ from $a$.
    \end{proof}

       Let $D$ be the component of $G- S$ containing $s_*$, which is attached to $S$ by \cref{lem:ExistenceMinimalSeparators}.
        Next, since $(C_*,S_*)$ is a side, there is a $v_*{-}s_*$ path $P$ in $C_*\cup S_*$ which meets $S_*$ precisely in its endpoint $s_*$.
        As $s_*\notin S$, \cref{claim:SepEmpty} implies that $P\cap S = \emptyset$.
        This shows that $C_*\subsetneq D$, and that $D$ is the connected component of $G-S$ containing $v_*$.
    As such, since both $D$ and the component of $G - S$ containing $a$ are attached to $S$ (by minimality of $S$ as an $s_*{-}a$ separator), \cref{lem:ExistenceMinimalSeparators} yields that $S$ is also a minimal $v_*{-}a$ separator, i.e.\ $S \in \mathcal{S}_{a}$. 
    Thus, we have $(D, S) \in \mathcal{C}$ but $(C_*,S_*)\prec (D,S)$, contradicting the $\preceq{-}$maximality of the side $(C_*,S_*)$.
        \end{proof}

The proof of \cref{thm:TreeDecompositionConstruction} will conclude by an iterative application of \cref{lem:finiteCliqueExt} above.
We recall that, if a given graph $G$ admits a normal spanning tree $T$ and $\leq_T$ denotes the tree-order of $T$ (with respect to an arbitrary root), then every connected induced subgraph of $G$ admits a $\leq_T{-}$minimal element (see \cite[Lemma 1.5.4]{bibel}).

\begin{proof}[Proof of \cref{thm:TreeDecompositionConstruction}]  
The forwards direction of \cref{thm:TreeDecompositionConstruction} follows by \cref{lem:Necessity}.
For the backwards direction, fix a chordal graph $G$ that does not contain $\obsii$ as an induced minor, but that admits a normal spanning tree $T_N$.
We denote by $\leq$ the corresponding tree-order of $T_N$, and denote the root of $T_N$ by $r$.

Our aim is then to obtain a tree-decomposition $(T,\cV)$ of $G$ into finite cliques.
Similar to the strategy in \cref{sec_2} above, we shall construct $(T,\cV)$ level by level. 
We do this by constructing induced subgraphs $G^0 \subset G^1 \subset \dots$ of $G$ together with tree-decompositions $(T^i, \cV^i)$ of $G^i$ into finite cliques, so that the tree-decompositions extend each other in the sense that
    \begin{enumerate}[label=(\roman*)]
    \item \label{item:increasingTd} $T^0 \subseteq T^1 \subseteq T^2 \subseteq \dots$,
    \item  \label{item:bagswelldefined} $V_t^i = V_t^j$ for all $i \leq j$ and $t \in T^i$.
    \end{enumerate}
    In particular, from now on we simply write $V_t$ for the part corresponding to a vertex $t \in T^i$. 
    To keep the construction going, we also ensure inductively that
    \begin{enumerate}[resume, label=(\roman*)]
    \item \label{item:compLeaf} for every component $C$ of $G-G^i$, its neighbourhood $N(C)$ is included in a part $V_{t_C}$ for a unique leaf $t_C \in T^i$ at height $i$.
    \end{enumerate}
    
    For $i=0$, let $(T^0,\mathcal{V}^0)$ be the trivial tree-decomposition of the subgraph $G^0:=\{r\}$ comprising only the root of $T_N$.

  Next, assume that for some $i \in \N$ we have already constructed a tree-decomposition~$(T^i, \cV^i)$ of $G^i$ into finite cliques. Let $C$ be a component of $G-G^i$ and write $S_C = N(C)$. By \cref{item:compLeaf}, the set $S_C$ is included in a (finite) clique $V_{t_C}$ for a unique leaf $t_C \in T^i$ at height $i$.
  Let $v_C$ denote the unique minimal element of $V(C)$ regarding the normal tree-order $\leq$. 
  Then $(C,S_C)$ is a side in $G$.
  By \cref{lem:finiteCliqueExt}, there exists a finite clique $B_C \subseteq G[S_C \cup C]$ with $S_C \subseteq B_C$ that satisfies the following property: 
  \begin{enumerate}[label=$(*)$]
      \item \label{item:starProp} either $v_C\in B_C$ or the connected component $G - B_C$ containing $v_C$ is unattached to $S_C$.
    \end{enumerate}
    Then, we declare $T^{i+1}$ to be the extension of $T^i$ in which, for every such component $C$ of $G-G^i$, we add a successor $s_C$ to $t_C$, with corresponding bag $V_{s_C} := V(B_C)$. 
    Let $G^{i+1}$ be the subgraph of $G$ induced by all parts of $T^{i+1}$. It is clear that $(T^{i+1}, \cV^{i+1})$ is a tree-decomposition of $G^{i+1}$ satisfying \cref{item:increasingTd}--\cref{item:compLeaf}. 
 
      The desired tree-decomposition\ satisfying the premises of \cref{thm:TreeDecompositionConstruction} will be $T := \bigcup_{i \in \N} T^i$ together with the parts $V_t$ for $t \in T$.
      These definitions are well-defined due to \cref{item:increasingTd} and \cref{item:bagswelldefined}. 
      
      Clearly, $(T, \cV)$ is a tree-decomposition\ of $\bigcup_{i \in \N} G^i$. To complete the proof, it suffices to prove that $\bigcup_{i \in \N} G^i = G$.
    Suppose for a contradiction that $\bigcup_{i \in \N} G^i \neq G$. 
    Then there is a component $D$ of $G - \bigcup_{i \in \N} G^i$.
     Let $v$ denote the unique minimal element of $V(D)$ regarding the normal tree-order $\leq$.
    For each $i \in \N$, let $C_i$ be the unique component of $G - G^i$ with $C_i \supseteq D$.

   We write $t_{i}:=t_{C_i}$, $S_{i}:=S_{C_i}$, $v_i:=v_{C_i}$ and $B_i:=B_{C_i}$ as in the definition of $(T^{i+1},\mathcal{V}^{i+1})$.
    By construction, we have $C_{i+1}\subsetneq C_i$ for every $i\in \mathbb{N}$, as well as $t_i<t_{i+1}$ in the tree-order of $T$. This implies that $v_{i}\leq v_{i+1}\leq v$, and hence there exists $k\in \mathbb{N}$ such that $v_{i} = v$ for all $i\geq k$.
    
    This means that $v_{i}\notin B_{i}$ for every $i\geq k$ and, hence, $C_{i+1}$ is unattached to $S_{i}$ for all $i\geq k$ by property \cref{item:starProp}. 
   Hence, we may choose a vertex $u_i \in S_i \setminus N(C_{i+1})$ for each $i \geq k$. 
       Since $S_i \subseteq B_i$ and $S_{i+1}\subseteq V_{t_{i+1}} = B_{i}$ by construction, it follows that $u_i,u_{i+1}\in B_{i}$ and, thus, either $u_i=u_{i+1}$ or $u_iu_{i+1}\in E(G)$.
    In particular, the graph $H$ induced by $\{u_i \colon i\geq k\}$ is connected.

    As in the earlier proof of \cref{thm_main_chordal}, the set $N(D)$ is an infinite clique.
    Since for every $x\in N(D)$ it holds that $x \in B_{i}$ for all sufficiently large $i$, we also have $xu_i\in E(G)$. 
    Hence, there are infinitely many independent $H$--$D$ paths in $G$.
    Moreover, $N(D)\cap H = \emptyset$ by the choice of each $u_i\in H$.
    Hence, $A:=V(H)$ and $B:=V(D)$ are disjoint connected sets of vertices in $G$ such that there is no edge of $G$ with one endvertex in $A$ and the other in $B$.
    Thus, by \cref{lem:ExistenceMinimalSeparators}, there exists a minimal $H{-}D$ separator $S$ that is disjoint from both $H$ and $D$.
    Since we have infinitely many $H$--$D$ paths, it follows that $S$ must be infinite.
    But then \cref{cor:H2obstruction} shows the existence of $\obsii$ as an induced minor of $G$, giving the desired contradiction.
\end{proof}

\bibliographystyle{amsplain}

\bibliography{collective.bib}

\end{document}